\documentclass[12pt,english]{article}
\usepackage{amssymb,amsthm}% Packages amssymb and amsthm are already loaded in e-jc. 
\usepackage{amsmath,amssymb}
\usepackage{graphicx}
\usepackage[colorlinks=true,citecolor=black,linkcolor=black,urlcolor=blue]{hyperref}

\usepackage{mathtools}
\usepackage{amsrefs}
\usepackage[normalem]{ulem}
\usepackage{comment}
\usepackage{bm}
\usepackage{tikz}
\usetikzlibrary{arrows,shapes}
\usepackage{colortbl}
\usepackage{xparse}
\usepackage{caption}
\usepackage{microtype}

\definecolor{mhcblue}{HTML}{0077CC} %official Mount Holyoke color
\definecolor{davidsonred}{HTML}{AC1A2F} %official Davidson College color

\theoremstyle{plain}
\newtheorem{thm}{\protect\theoremname}
\theoremstyle{plain}
\newtheorem{lem}[thm]{\protect\lemmaname}
\theoremstyle{plain}
\newtheorem{cor}[thm]{\protect\corollaryname}
\theoremstyle{plain}
\newtheorem{prop}[thm]{\protect\propositionname}

\makeatletter

\providecommand{\tabularnewline}{\\}

\DeclareMathOperator{\pk}{\mathsf{pk}}
\DeclareMathOperator{\exc}{\mathsf{exc}}
\DeclareMathOperator{\des}{\mathsf{des}}
\DeclareMathOperator{\Dasc}{\mathsf{Dasc}}
\DeclareMathOperator{\Ddes}{\mathsf{Ddes}}
\DeclareMathOperator{\st}{\mathsf{st}}
\DeclareMathOperator{\val}{\mathsf{val}}
\DeclareMathOperator{\dasc}{\mathsf{dasc}}
\DeclareMathOperator{\ddes}{\mathsf{ddes}}
\DeclareMathOperator{\dbl}{\mathsf{dbl}}
\DeclareMathOperator{\cdbl}{\mathsf{cdbl}}
\DeclareMathOperator{\cpk}{\mathsf{cpk}}
\DeclareMathOperator{\cval}{\mathsf{cval}}
\DeclareMathOperator{\cddes}{\mathsf{cddes}}
\DeclareMathOperator{\cdasc}{\mathsf{cdasc}}
\DeclareMathOperator{\fix}{\mathsf{fix}}
\DeclareMathOperator{\cyc}{\mathsf{cyc}}

\usepackage{titlesec}
\titleformat{\section}
  {\normalfont\large\bfseries}{\thesection.}{.8ex}{} %changed from Large
\titleformat{\subsection}
  {\normalfont\bfseries}{\thesubsection.}{.8ex}{} %large removed
\titleformat{\subsubsection}
  {\normalfont\bfseries}{\thesubsubsection.}{.8ex}{}

\let\originalleft\left
\let\originalright\right
\renewcommand{\left}{\mathopen{}\mathclose\bgroup\originalleft}
\renewcommand{\right}{\aftergroup\egroup\originalright}

\usepackage{babel}
\providecommand{\corollaryname}{Corollary}
\providecommand{\lemmaname}{Lemma}
\providecommand{\propositionname}{Proposition}
\providecommand{\theoremname}{Theorem}

\title{Hopping from Chebyshev polynomials\\ to permutation statistics}

\author{Jordan O.\ Tirrell\\
	\small Department of Mathematics and Statistics\\[-0.8ex]
	\small Mount Holyoke College\\[-0.8ex] 
	\small South Hadley, MA, U.S.A.\\
	\small\tt jtirrell@mtholyoke.edu\\
	\and
	Yan Zhuang\\
	\small Department of Mathematics and Computer Science\\[-0.8ex]
	\small Davidson College\\[-0.8ex]
	\small Davidson, NC, U.S.A.\\
	\small\tt yazhuang@davidson.edu}

\begin{document}

\maketitle

% E-JC papers must include an abstract. The abstract should consist of a
% succinct statement of background followed by a listing of the
% principal new results that are to be found in the paper. The abstract
% should be informative, clear, and as complete as possible. Phrases
% like "we investigate..." or "we study..." should be kept to a minimum
% in favor of "we prove that..."  or "we show that...".  Do not
% include equation numbers, unexpanded citations (such as "[23]"), or
% any other references to things in the paper that are not defined in
% the abstract. The abstract may be distributed without the rest of the
% paper so it must be entirely self-contained.  Try to include all words
% and phrases that someone might search for when looking for your paper.
\begin{abstract}
We prove various formulas which express exponential generating functions
counting permutations by the peak number, valley number, double ascent
number, and double descent number statistics in terms of the exponential
generating function for Chebyshev polynomials, as well as cyclic analogues
of these formulas for derangements. We give several applications of
these results, including formulas for the $(-1)$-evaluation of some
of these distributions. Our proofs are combinatorial and 
involve the use of monomino-domino tilings, the modified
Foata\textendash Strehl action (a.k.a.\ valley-hopping), and a cyclic analogue of this action due to Sun and Wang.
\end{abstract}
%\textbf{Keywords: }{permutation statistics, Chebyshev polynomials, Pell numbers, modified Foata--Strehl action, derangements}
%{\let\thefootnote\relax\footnotetext{2010 \textit{Mathematics Subject Classification}. Primary 05A15; Secondary 05A05, 33C45.}}

\pagebreak[2]

\section{Introduction}

Let $\pi=\pi_{1}\pi_{2}\cdots\pi_{n}$ be a permutation (written in
one-line notation) in $\mathfrak{S}_{n}$, the set of permutations
of $[n]=\{1,2,\dots,n\}$. We say that $\pi_{i}$ (where $i\in[n-1]$)
is a \textit{descent} if $\pi_{i}>\pi_{i+1}$, and that $\pi_{i}$
(where $2\leq i\leq n-1$) is a \textit{peak} of $\pi$ if $\pi_{i-1}<\pi_{i}>\pi_{i+1}$.
Define $\des(\pi)$ to be the number of descents of $\pi$ and $\pk(\pi)$
to be the number of peaks of $\pi$. The descent number $\des$ and
peak number $\pk$ are classical permutation statistics whose study
dates back to MacMahon \cite{macmahon} and to David and Barton \cite{David1962},
respectively.

The $n$th \textit{Eulerian polynomial}\footnote{We note that many works instead define the $n$th Eulerian polynomial to be
	$\sum_{\pi\in\mathfrak{S}_{n}}t^{\des(\pi)+1}$.}
\[
A_{n}(t)\coloneqq\sum_{\pi\in\mathfrak{S}_{n}}t^{\des(\pi)}
\]
encodes the distribution of the descent number $\des$ over $\mathfrak{S}_{n}$, and the $n$th \textit{peak polynomial } 
\[
P_{n}^{\pk}(t)\coloneqq\sum_{\pi\in\mathfrak{S}_{n}}t^{\pk(\pi)}
\]
is the analogous polynomial for the peak number $\pk$. 

It is well-known \cite[Th\'{e}or\`{e}me 5.6]{Foata1970} that the
$(-1)$-evaluation of the Eulerian distribution is given by the formula
\begin{equation}
A_{n}(-1)=\begin{cases}
(-1)^{(n-1)/2}E_{n}, & \mbox{if }n\text{ is odd,}\\
0, & \mbox{if }n\text{ is even,}
\end{cases}\label{e-A-1}
\end{equation}
where $E_{n}$ is the $n$th \textit{Euler number} defined by 
\[
\sum_{n=0}^{\infty}E_{n}\frac{x^{n}}{n!}=\sec(x)+\tan(x).
\]
(The Euler numbers $E_{n}$ for odd $n$ are called \textit{tangent
numbers}, and those for even $n$ are called \textit{secant numbers}.)
No combinatorial formula for $P_{n}^{\pk}(-1)$ is known, although
this sequence does appear on the OEIS \cite[A006673]{oeis}. The first
several terms of this sequence are given in the following table:
\noindent \begin{center}
\begin{tabular}{c|cccccccccc}
$n$ & 1 & 2 & 3 & 4 & 5 & 6 & 7 & 8 & 9 & 10\tabularnewline
\hline 
$P_{n}^{\pk}(-1)$ & $1$ & $2$ & $2$ & $-8$ & $-56$ & $-112$ & $848$ & $9088$ & $25216$ & $-310528$\tabularnewline
\end{tabular}
\end{center}
We note that the apparent 6-periodicity of the sequence of signs breaks at $n=42$. Very recently, Troyka \cite{Troyka2019} argued that there is no $k$ for which the sequence of signs of the $P_{n}^{\pk}(-1)$ is $k$-periodic, which suggests that there is unlikely to be a nice combinatorial interpretation for the $P_{n}^{\pk}(-1)$.

The exponential generating functions for $A_{n}(t)$ and $P_{n}^{\pk}(t)$
have the following well-known expressions:\footnote{These exponential generating functions are usually given a constant
term of 1 in the literature, but it is more convenient to define these
without the constant term in this work.}
\[
A(t;x)\coloneqq\sum_{n=1}^{\infty}A_{n}(t)\frac{x^{n}}{n!}=\frac{e^{(1-t)x}-1}{1-te^{(1-t)x}};
\]
\[
P^{\pk}(t;x)\coloneqq\sum_{n=1}^{\infty}P_{n}^{\pk}(t)\frac{x^{n}}{n!}=\frac{1}{\sqrt{1-t}\coth(x\sqrt{1-t})-1}.
\]
The work in this paper was originally inspired by the curious observation
that $A(-1;x)$ and $P^{\pk}(-1;x)$ can be expressed as the logarithmic
derivative of the exponential generating function of some non-negative
integer sequence. For the Eulerian polynomials, this sequence $\{f_{n}\}_{n\geq0}$
is simply $f_{n}\coloneqq(n+1)\negthickspace\mod2$, i.e., the sequence
$1,0,1,0,\dots$, whose exponential generating function is given by 
\[
F(x)\coloneqq1+\frac{x^{2}}{2!}+\frac{x^{4}}{4!}+\cdots=\cosh(x).
\]
For the peak polynomials, this sequence is the sequence of Pell numbers,
which has been widely studied in combinatorics and number theory.
The Pell numbers $\{g_{n}\}_{n\geq0}$ are defined by the recursive
formula $g_{n}\coloneqq 2g_{n-1}+g_{n-2}$ for $n\geq2$ with initial values
$g_{0}=1$ and $g_{1}=0$. The first several terms of
this sequence are below:\renewcommand{\arraystretch}{1.2}
\begin{center}
\begin{tabular}{c|ccccccccccccc}
$n$ & 0 & 1 & 2 & 3 & 4 & 5 & 6 & 7 & 8 & 9 & 10 & 11 & 12\tabularnewline
\hline 
$g_{n}$ & 1 & 0 & 1 & 2 & 5 & 12 & 29 & 70 & 169 & 408 & 985 & 2378 & 5741\tabularnewline
\end{tabular}
\end{center}

\noindent Note that the indexing here is slightly different from the
usual indexing of the Pell numbers (see OEIS \cite[A000129]{oeis}).
The exponential generating function of $\{g_{n}\}_{n\geq0}$ is given
by 

\[
G(x)\coloneqq\sum_{n=0}^{\infty}g_{n}\frac{x^{n}}{n!}=\frac{1}{2}e^{x}(2\cosh(x\sqrt{2})-\sqrt{2}\sinh(x\sqrt{2})).
\]
\begin{thm}
\label{t-main} The exponential generating functions for the Eulerian
and peak polynomials evaluated at $t=-1$ can be expressed as the
logarithmic derivative of $F(x)$ and $G(x)$, respectively. That
is:
\begin{enumerate}
\item [\normalfont{(a)}] ${\displaystyle {\displaystyle A(-1;x)=\frac{\frac{d}{dx}F(x)}{F(x)}}}$
\item [\normalfont{(b)}] ${\displaystyle P^{\pk}(-1;x)=\frac{\frac{d}{dx}G(x)}{G(x)}}$
\end{enumerate}
\end{thm}

While Theorem \ref{t-main} can be proven directly by algebraically
manipulating the generating function formulas for $A(t;x)$, $P^{\pk}(t;x)$,
$F(x)$, and $G(x)$, one of our goals in this paper is to present
a combinatorially-flavored proof. In Section 2, we define several
other relevant permutation statistics and introduce a key ingredient
of our proof: the modified Foata\textendash Strehl group action (a.k.a.\ valley-hopping).
In Section 3, we define a two-parameter variant of the Chebyshev polynomials
of the second kind which specialize to both the numbers $f_{n}$ and
the Pell numbers $g_{n}$.
Like the ordinary Chebyshev polynomials of the second kind, our bivariate Chebyshev polynomials have as a combinatorial model monomino-domino tilings of a rectangle, but with slightly different weights. We present a formula (Theorem \ref{t-pkdblV})
involving these Chebyshev polynomials for the joint distribution of two statistics: the peak number, and the total number of double
ascents and double descents. We give a combinatorial proof of Theorem
\ref{t-pkdblV} which involves tilings and valley-hopping, and a special case of this result
implies Theorem \ref{t-main} (b). We transform Theorem 3 into similar
results for other permutation statistics, which we then use to prove
Theorem \ref{t-main} (a) and to prove that the $(-1)$-evaluation
of double descent distributions yields the tangent numbers. 

In Section 4, we turn our attention to counting derangements by cyclic
analogues of the permutation statistics studied in Sections 2\textendash 3.
Using a variant of valley-hopping due to Sun and Wang \cite{Sun2014}
for derangements, we prove a cyclic analogue of Theorem \ref{t-pkdblV}
and use it to derive formulas relating exponential generating
functions counting derangements by cyclic statistics with the exponential generating function for our Chebyshev polynomials. We use this to prove a result
similar to Theorem \ref{t-main} for the excedance and cyclic peak
distributions over derangements, and to prove that the $(-1)$-evaluation
of cyclic double descent distributions over derangements yields
the secant numbers.

\section{Permutation statistics and valley-hopping}

Given a permutation $\pi=\pi_{1}\pi_{2}\cdots\pi_{n}$ in $\mathfrak{S}_{n}$,
we say that $\pi_{i}$ (where $i\in[n]$) is:
\begin{itemize}
\item a \textit{valley} if $\pi_{i-1}>\pi_{i}<\pi_{i+1}$;
\item a \textit{double ascent} if $\pi_{i-1}<\pi_{i}<\pi_{i+1}$;
\item a \textit{double descent} if $\pi_{i-1}>\pi_{i}>\pi_{i+1}$;
\end{itemize}
where we are using the convention $\pi_{0}=\pi_{n+1}=\infty$.\footnote{We note that many works on permutation enumeration do not use these conventions, and simply restrict the possible positions of valleys, double ascents, and double descents to the interval from $2$ to $n-1$. What we call ``valleys'' are sometimes called ``left-right valleys'' or ``exterior valleys'', what we call ``double ascents'' are sometimes called ``right double ascents'', and what we call ``double descents'' are sometimes called ``left double descents''. (See, e.g., \cite{Zhuang2016}.)} Thus,
every letter of a permutation is either a peak, valley, double ascent,
or double descent. Define $\val(\pi)$, $\dasc(\pi)$, and $\ddes(\pi)$
to be the number of valleys, double ascents, and double descents of
$\pi$, respectively. We also define $\dbl(\pi)\coloneqq\dasc(\pi)+\ddes(\pi)$
to be the total number of double ascents and double descents of $\pi$.

For a list of statistics $\st_{1},\st_{2},\dots,\st_{m}$
and corresponding variables $t_{1},t_{2},\dots,t_{m}$, we define
the polynomials $\{P_{n}^{(\st_{1},\st_{2},\dots\st_{m})}(t_{1},t_{2},\dots,t_{m})\}_{n\geq0}$
by 
\[
P_{n}^{(\st_{1},\st_{2},\dots\st_{m})}(t_{1},t_{2},\dots,t_{m})\coloneqq\sum_{\pi\in\mathfrak{S}_{n}}t_{1}^{\st_{1}(\pi)}t_{2}^{\st_{2}(\pi)}\cdots t_{m}^{\st_{m}(\pi)}.
\]
and we let 
\[
P^{(\st_{1},\st_{2},\dots\st_{m})}(t_{1},t_{2},\dots,t_{m};x)\coloneqq\sum_{n=1}^{\infty}P_{n}^{(\st_{1},\st_{2},\dots\st_{m})}(t_{1},t_{2},\dots,t_{m})\frac{x^{n}}{n!}
\]
be their exponential generating function.\footnote{In the case where we have a single statistic $\st$, we write these
simply as $P_{n}^{\st}(t)$ and $P^{\st}(t;x)$.} For example, we have 
\[
P_{n}^{(\pk,\dbl)}(s,t)=\sum_{\pi\in\mathfrak{S}_{n}}s^{\pk(\pi)}t^{\dbl(\pi)}
\]
and 
\[
P^{(\pk,\dbl)}(s,t;x)=\sum_{n=1}^{\infty}P_{n}^{(\pk,\dbl)}(s,t)\frac{x^{n}}{n!};
\]
we will consider these on the way to proving Theorem \ref{t-main}.
Our proof will make use of a bijection based on a group action on
$\mathfrak{S}_{n}$ induced by involutions which toggle between double
ascents and double descents; we will spend the remainder of this section
defining this action and the associated bijection.

For $\pi\in\mathfrak{S}_{n}$, fix $k\in[n]$. We may write $\pi=w_{1}w_{2}kw_{4}w_{5}$
where $w_{2}$ is the maximal consecutive subword immediately to the
left of $k$ whose letters are all smaller than $k$, and $w_{4}$
is the maximal consecutive subword immediately to the right of $k$
whose letters are all smaller than $k$. For example, if $\pi=467125839$
and $k=5$, then $\pi$ is the concatenation of $w_{1}=467$, $w_{2}=12$,
$k=5$, the empty word $w_{4}$, and $w_{5}=839$.

Define $\varphi_{k}\colon\mathfrak{S}_{n}\rightarrow\mathfrak{S}_{n}$
by 
\[
\varphi_{k}(\pi)=\begin{cases}
w_{1}w_{4}kw_{2}w_{5}, & \mbox{if }k\mbox{ is a double ascent or double descent of \ensuremath{\pi},}\\
\pi, & \mbox{if }k\mbox{ is a peak or valley of \ensuremath{\pi}.}
\end{cases}
\]
Equivalently, $\varphi_{k}(\pi)=w_{1}w_{4}kw_{2}w_{5}$ if exactly
one of $w_{2}$ and $w_{4}$ is nonempty, and $\varphi_{k}(\pi)=\pi$
otherwise. For any subset $S\subseteq[n]$, we define $\varphi_{S}\colon\mathfrak{S}_{n}\rightarrow\mathfrak{S}_{n}$
by $\varphi_{S}=\prod_{k\in S}\varphi_{k}$. It is easy to see that
$\varphi_{S}$ is an involution, and that for all $S,T\subseteq[n]$,
the involutions $\varphi_{S}$ and $\varphi_{T}$ commute with each
other. Hence the involutions $\{\varphi_{S}\}_{S\subseteq[n]}$ define
a $\mathbb{Z}_{2}^{n}$-action on $\mathfrak{S}_{n}$ which is often
called the \textit{modified Foata\textendash Strehl action} or \textit{valley-hopping}.
This action is based on a classical group action of Foata and Strehl
\cite{Foata1974}, was introduced by Shapiro, Woan, and Getu \cite{Shapiro1983}, and was later rediscovered by Br\"and\'en \cite{Braenden2008}.

Let $\widetilde{\mathfrak{S}}_{n}$ denote the set of permutations of
$[n]$ with no double ascents and where each double descent is assigned
one of two colors: red or blue.\footnote{To be more precise, \textcolor{davidsonred}{Davidson College red} or \textcolor{mhcblue}{Mount Holyoke College blue}.} Then valley-hopping induces a map
$\Phi$ from $\widetilde{\mathfrak{S}}_{n}$ to $\mathfrak{S}_{n}$ defined
in the following way. Given a permutation $\pi$ in $\widetilde{\mathfrak{S}}_{n}$,
let $R(\pi)$ be the set of red double descents in $\pi$ and let
$\bar{\pi}$ be the corresponding permutation of $\pi$ in $\mathfrak{S}_{n}$,
that is, the permutation obtained by forgetting the colors on the
double descents. Then let $\Phi(\pi)=\varphi_{R(\pi)}(\bar{\pi})$.
For example, if $\pi = \textcolor{davidsonred}{7}26\textcolor{mhcblue}{5}39\textcolor{davidsonred}{8}\textcolor{mhcblue}{4}1$,
then $\Phi(\pi)=265379418$. (See Figure 1.)
\noindent \begin{center}
\begin{figure}
\begin{center}
\begin{tikzpicture}[scale=0.5] 	
\draw[step=1,lightgray,thin] (0,1) grid (10,10); 
	\tikzstyle{ridge}=[draw, line width=1, dotted, color=black] 
	\path[ridge] (0,10)--(1,7)--(2,2)--(3,6)--(4,5)--(5,3)--(6,9)--(7,8)--(8,4)--(9,1)--(10,10); 
	\tikzstyle{node0}=[circle, inner sep=2, fill=black] 
	\tikzstyle{node1}=[rectangle, inner sep=3, fill=mhcblue] 
	\tikzstyle{node2}=[diamond, inner sep=2, fill=davidsonred] 
	\node[node0] at (0,10) {}; 
	\node[node2] at (1,7) {}; 
	\node[node0] at (2,2) {}; 
	\node[node0] at (3,6) {}; 
	\node[node1] at (4,5) {}; 
	\node[node0] at (5,3) {}; 
	\node[node0] at (6,9) {}; 
	\node[node2] at (7,8) {}; 
	\node[node1] at (8,4) {}; 
	\node[node0] at (9,1) {}; 
	\node[node0] at (10,10) {}; 
	\tikzstyle{hop}=[draw, line width = 1.5, color=davidsonred,->] 
	\path[hop] (1.5,7)--(5.3,7); 
	\path[hop] (7.5,8)--(9.5,8); 
	\tikzstyle{pi}=[above=-1] 
	\node[pi] at (0,0) {$\infty$}; 
	\node[pi, color=davidsonred] at (1,0) {7}; 
	\node[pi] at (2,0) {2}; 
	\node[pi] at (3,0) {6}; 
	\node[pi, color=mhcblue] at (4,0) {5}; 
	\node[pi] at (5,0) {3}; 
	\node[pi] at (6,0) {9}; 
	\node[pi,color=davidsonred] at (7,0) {8}; 
	\node[pi,color=mhcblue] at (8,0) {4}; 
	\node[pi] at (9,0) {1}; 
	\node[pi] at (10,0) {$\infty$}; 
	\path[draw,line width=1,->] (11,5)--(15,5); 
	\begin{scope}[shift={(16,0)}] 
	\draw[step=1,lightgray,thin] (0,1) grid (10,10); 
	\path[ridge] (0,10)--(1,2)--(2,6)--(3,5)--(4,3)--(5,7)--(6,9)--(7,4)--(8,1)--(9,8)--(10,10); 
	\node[node0] at (0,10) {}; 
	\node[node0] at (1,2) {}; 
	\node[node0] at (2,6) {}; 
	\node[node0] at (3,5) {}; 
	\node[node0] at (4,3) {}; 
	\node[node0] at (5,7) {}; 
	\node[node0] at (6,9) {}; 
	\node[node0] at (7,4) {}; 
	\node[node0] at (8,1) {}; 
	\node[node0] at (9,8) {}; 
	\node[node0] at (10,10) {}; 
	\node[pi] at (0,0) {$\infty$}; 
	\node[pi] at (1,0) {2}; 
	\node[pi] at (2,0) {6}; 
	\node[pi] at (3,0) {5}; 
	\node[pi] at (4,0) {3}; 
	\node[pi] at (5,0) {7}; 
	\node[pi] at (6,0) {9};
	\node[pi] at (7,0) {4}; 
	\node[pi] at (8,0) {1}; 
	\node[pi] at (9,0) {8}; 
	\node[pi] at (10,0) {$\infty$}; 
	\end{scope}
\end{tikzpicture}
\end{center}

\caption{Valley-hopping induces a bijection between $\widetilde{\mathfrak{S}}_{n}$
and $\mathfrak{S}_{n}$.}
\end{figure}
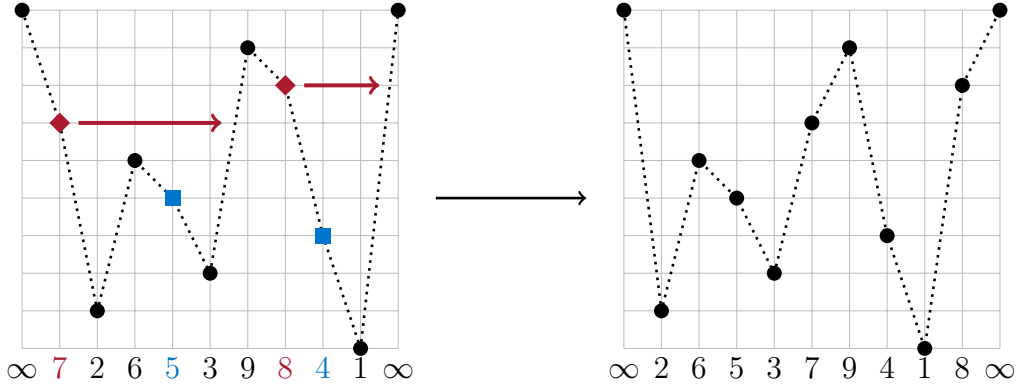
\end{center}
\begin{lem}
\label{l-pkdblbij} The map $\Phi\colon\widetilde{\mathfrak{S}}_{n}\rightarrow\mathfrak{S}_{n}$
is a $(\pk,\dbl)$-preserving bijection.
\end{lem}

\begin{proof}
The inverse $\Phi^{-1}$ of the map $\Phi$ can be described in the
following way. Let $\Dasc(\pi)$ be the set of double ascents of $\pi$
and $\Ddes(\pi)$ the set of double descents of $\pi$. If $S\subseteq\Ddes(\pi)$
and if $\pi$ has no double ascents, then let $\pi^{S}$ be the permutation
in $\widetilde{\mathfrak{S}}_{n}$ obtained by coloring the double descents
in $S$ blue and all other double descents red. Given a permutation
$\pi$ in $\mathfrak{S}_{n}$, let $\Phi^{-1}(\pi)=(\varphi_{\Dasc(\pi)}(\pi))^{\Ddes(\pi)}$.
Then $\Phi$ is a bijection between $\widetilde{\mathfrak{S}}_{n}$ and
$\mathfrak{S}_{n}$. The claim that $\Phi$ preserves the $\pk$ and
$\dbl$ statistics follows from the easy fact that valley-hopping
preserves these statistics as well.
\end{proof}

\section{\label{s-hopping}Hopping from Chebyshev polynomials to permutation
statistics}

The Chebyshev polynomials of the second kind $\{U_n(t)\}_{n\geq0}$ are defined by the recurrence $U_n(t)\coloneqq 2tU_{n-1}(t)-U_{n-2}(t)$ for $n\geq 2$ with initial values $U_0(t)=1$ and $U_1(t)=2t$. These polynomials are an important sequence of orthogonal polynomials arising in many branches of mathematics; see \cite{Andrews1999,Rain60,Rivlin1990} for several notable references.

It will be more convenient for us to use a two-parameter variant $\{U_n(s,t)\}_{n\geq0}$ of the Chebyshev polynomials of the second kind. We define $U_n(s,t)$ by the recurrence
\begin{equation}\label{e-Ust_rec}
	U_n(s,t)=2tU_{n-1}(s,t)-sU_{n-2}(s,t)
\end{equation}
for $n\geq2$ with initial values $U_0(s,t)=1$ and $U_1(s,t)=2t$. The first several of these bivariate Chebyshev polynomials are given in the following table.
\begin{center}
\begin{tabular}{ccc}
	\begin{tabular}{c|c}
		$n$ & $U_n(s,t)$\tabularnewline
		\hline 
		0 & 1\tabularnewline
		1 & $2t$\tabularnewline
		2 & $4t^2-s$\tabularnewline
		3 & $8t^3-4st$\tabularnewline
	\end{tabular}
	& &
	\begin{tabular}{c|c}
		$n$ & $U_n(s,t)$\tabularnewline
		\hline 
		4 & $16t^4-12st^2+s^2$\tabularnewline
		5 & $32t^5-32st^3+6s^2t$\tabularnewline
		6 & $64t^6-80st^4+24s^2t^2-s^3$\tabularnewline
		7 & $128t^7-192st^5+80s^2t^3-8s^3t$\tabularnewline
	\end{tabular}
\end{tabular}
\end{center}

Note that these polynomials are related to the usual Chebyshev polynomials of the second kind by the formulas $U_n(t)=U_n(1,t)$ and $U_n(s,t)=U_n(s^{-1/2}t)s^{n/2}$. The numbers $f_{n}=(n+1)\negthickspace\mod2$ and the Pell numbers $g_{n}$ are specializations of the $U_n(s,t)$, as $f_{n}=U_{n-2}(-1,0)$ and $g_{n}=U_{n-2}(-1,1)$ for all $n\geq2$. From the recurrence (\ref{e-Ust_rec}), it is not hard to see that the ordinary generating function for our $U_{n}(s,t)$ is given by the formula 
\[
\sum_{n=0}^{\infty}U_{n}(s,t)x^{n}=\frac{1}{1-2tx+sx^{2}},
\]
and that $U_{n}(s,t)$ counts tilings of a $1\times n$ rectangle with two types of monominoes, each weighted $t$, and one type of domino, each weighted $-s$. (The ordinary Chebyshev polynomials of the second kind count the same tilings but with dominoes weighted $-1$; see \cite{Benjamin2009} for an accessible reference.)

An expression for the exponential generating function of the $U_n(t)$ is known (see~\cite[p.~301]{Rain60}) and together with the formula $U_n(s,t)=U_n(s^{-1/2}t)s^{n/2}$, we obtain 
\begin{equation}\label{e-Uegf}
\sum_{n=0}^\infty U_n(s,t)\frac{x^{n+1}}{(n+1)!}=e^{xt}\frac{\sinh{(x\sqrt{t^2-s})}}{\sqrt{t^2-s}}.
\end{equation}
We will find it more convenient to work with the exponential generating function
\[
V(s,t;x)\coloneqq\sum_{n=0}^{\infty}U_{n}(s,t)\frac{x^{n+2}}{(n+2)!}=\frac{x^{2}}{2!}+2t\frac{x^{3}}{3!}+(4t^{2}-s)\frac{x^{4}}{4!}+\cdots.
\]
Note that
\[
F(x)=\sum_{n=0}^{\infty}f_{n}\frac{x^{n}}{n!}=1+V(-1,0;x)
\quad \mathrm{and} \quad
G(x)=\sum_{n=0}^{\infty}g_{n}\frac{x^{n}}{n!}=1+V(-1,1;x).
\]
It follows from~\eqref{e-Uegf} that $V(s,t;x)$ has the closed-form expression
\begin{equation}
V(s,t;x)=\frac{1}{s}\left(1-\cosh(x\sqrt{t^{2}-s})e^{tx}+\frac{te^{tx}\sinh(x\sqrt{t^{2}-s})}{\sqrt{t^{2}-s}}\right).\label{e-Vcf}
\end{equation}

\subsection{A Chebyshev formula for the bidistribution \texorpdfstring{$(\protect\pk,\protect\dbl)$}{(pk, dbl)}}

We now present our main theorem from this section.
\begin{thm}
\label{t-pkdblV} ${\displaystyle P^{(\pk,\dbl)}(s,t;x)=\frac{\frac{\partial}{\partial x}V(s,t;x)}{1-sV(s,t;x)}}$
\end{thm}

Setting $s=-1$ and $t=1$ in Theorem \ref{t-pkdblV} yields Theorem
\ref{t-main} (b). Observe that the numerator in Theorem~\ref{t-pkdblV} appears in Equation~\eqref{e-Uegf}, and using Equation~\eqref{e-Vcf} for the denominator we can obtain the expression 
\begin{align*} 
P^{(\pk,\dbl)}(s,t;x)=\frac1{\sqrt{t^2-s}\coth{(x\sqrt{t^2-s})}-t}.
\end{align*}
\begin{proof}
From the combinatorial interpretation of multiplication of exponential
generating functions (see, e.g., \cite[Proposition 5.1.3]{Stanley2001}),
it suffices to show that 
\begin{align}
P_{n}^{(\pk,\dbl)}(s,t) & =\sum_{k=0}^{\left\lfloor (n-1)/2\right\rfloor }s^{k}\sum_{B}U_{\left|B_{0}\right|-1}(s,t)U_{\left|B_{1}\right|-2}(s,t)\cdots U_{\left|B_{k}\right|-2}(s,t)\label{e-bigsum}
\end{align}
where the second sum is over all ordered set partitions $B$ of $[n]$
into blocks $B_{0},B_{1},\dots,B_{k}$ such that every block other than $B_{0}$ has size at least 2. Thus, the right-hand side
of Equation (\ref{e-bigsum}) counts these set partitions together
with:
\begin{itemize}
\item a tiling of an $1\times\left(\left|B_{0}\right|-1\right)$ rectangle
with two types of monominoes (colored red and blue), each weighted
$t$, and one type of domino, each weighted $-s$;
\item for each $1\leq i\leq k$, a tiling of an $1\times\left(\left|B_{i}\right|-2\right)$
rectangle with the same types of shapes and weights as above;
\end{itemize}
and each block (other than $B_{0}$) is given an additional weight
of $s$. We place an $\infty$ in the first block, write out each
block in decreasing order, and separate adjacent blocks with a bar,
as in
\[
\infty>\pi_{1}>\pi_{2}>\cdots>\pi_{\left|B_{0}\right|}\mid\pi_{\left|B_{0}\right|+1}>\cdots>\pi_{\left|B_{0}\right|+\left|B_{1}\right|}\mid\cdots\mid\pi_{n-\left|B_{k}\right|+1}>\cdots>\pi_{n}.
\]
Here, we consider the tiling on each block as being a tiling on all
but the first and last elements of the block. Now we define a sign-reversing
involution on these objects in the following way: Find the first pair
of elements $(\pi_{i},\pi_{i+1})$ where there is a domino, or where
$\pi_{i}$ and $\pi_{i+1}$ are in separate blocks and $\pi_{i}>\pi_{i+1}$.
If $(\pi_{i},\pi_{i+1}$) is covered by a domino, then we remove the
domino and insert a new bar in between $\pi_{i}$ and $\pi_{i+1}$,
thus splitting their block into two blocks. If $\pi_{i}$ and $\pi_{i+1}$
are in separate blocks and $\pi_{i}>\pi_{i+1}$, then we merge the
two blocks and cover $(\pi_{i},\pi_{i+1})$ with a domino. (See Figure
2.) This involution swaps a domino (weighted $-s$) with an additional
block (weighted $s$), and after cancellation we are left with those
objects with no dominoes and such that $\pi_{i}<\pi_{i+1}$ whenever
$\pi_{i}$ and $\pi_{i+1}$ are in separate blocks.

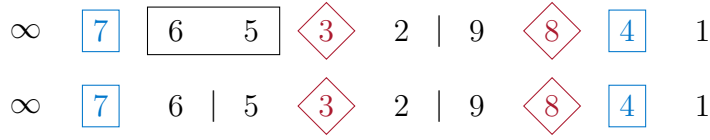
\begin{figure}
\begin{center}
	\begin{tikzpicture}[scale=1] 	
	\tikzstyle{mono1}=[draw,rectangle,inner sep=4, color=mhcblue] 	
	\tikzstyle{mono2}=[draw,diamond,inner sep=2, color=davidsonred] 
	\node at (0,0) {$\infty$}; 	
	\node[mono1] at (1,0) {7}; 
	\node at (2,0) {6}; 
	\node[draw,rectangle,minimum width=50,minimum height=18] at (2.5,0) {}; 
	\node at (3,0) {5}; 
	\node[mono2] at (4,0) {3}; 
	\node at (5,0) {2}; 
	\node at (5.5,0) {$|$}; 
	\node at (6,0) {9}; 
	\node[mono2] at (7,0) {8}; 
	\node[mono1] at (8,0) {4};
	\node at (9,0) {1}; 
	\begin{scope}[shift={(0,-1)}]
	\node at (0,0) {$\infty$}; 
	\node[mono1] at (1,0) {7}; 
	\node at (2,0) {6}; 
	\node at (2.5,0) {$|$}; 
	\node at (3,0) {5}; 
	\node[mono2] at (4,0) {3}; 
	\node at (5,0) {2}; 
	\node at (5.5,0) {$|$};
	\node at (6,0) {9}; 
	\node[mono2] at (7,0) {8};
	\node[mono1] at (8,0) {4};
	\node at (9,0) {1}; 
	\end{scope}
	\end{tikzpicture}
\end{center}
\caption{Example of the sign-reversing involution.}
\end{figure}

If we treat any one of these remaining objects $\pi=\pi_{1}\pi_{2}\cdots\pi_{n}$
as a permutation, we see that $\pi$ has no double ascents and has
each double descent colored either red or blue (depending on the color
of the corresponding monomino). Hence, $\pi$ belongs to $\widetilde{\mathfrak{S}}_{n}$
and contributes a weight of $s^{\pk(\pi)}t^{\dbl(\pi)}$ to the right-hand
side of Equation (\ref{e-bigsum}). The result then follows from applying
the $(\pk,\dbl)$-preserving bijection $\Phi$.
\end{proof}

\subsection{A Chebyshev formula for the quadruple distribution \texorpdfstring{$(\protect\pk,\protect\val,\protect\dasc,\protect\ddes)$}{(pk,val,dasc,ddes)}}

We shall now derive from Theorem \ref{t-pkdblV} an analogous result
for the joint distribution of the four statistics $\pk$, $\val$,
$\dasc$, and $\ddes$.
\begin{thm}
\label{t-pkvaldascddes} ${\displaystyle {\displaystyle P^{(\pk,\val,\dasc,\ddes)}(s,t,u,v;x)=\frac{t\frac{\partial}{\partial x}V(st,(u+v)/2;x)}{1-stV(st,(u+v)/2;x)}}}$
\end{thm}
\begin{proof}
First, observe that $\val(\pi)=\pk(\pi)+1$ for all permutations $\pi$,
and that the reversal $r(\pi)\coloneqq\pi_{n}\pi_{n-1}\cdots\pi_{1}$
of a permutation is a $(\pk,\val)$-preserving involution on $\mathfrak{S}_{n}$
that switches double ascents with double descents. Thus, we have
\[
P_{n}^{(\pk,\val,\dasc,\ddes)}(s,t,u,v)=tP_{n}^{(\pk,\dbl)}\left(st,\frac{1}{2}(u+v)\right)
\]
for all $n\geq1$, which proves the result in light of Theorem \ref{t-pkdblV}.
\end{proof}
Theorem \ref{t-pkvaldascddes} and the formula (\ref{e-Vcf}) for
$V(s,t;x)$ can be used together to derive the closed-form formula
\[
P^{(\pk,\val,\dasc,\ddes)}(s,t,u,v;x)=\frac{2t}{\alpha\coth(\frac{1}{2}\alpha x)-u-v}
\]
where $\alpha=\sqrt{(u+v)^{2}-4st}$. This is equivalent to a classical formula of Carlitz and Scoville \cite{Carlitz1974}; see also \cite[Exercise 1.61a]{Stanley2011}.

The following corollary states several specializations of Theorem \ref{t-pkvaldascddes}.

\begin{cor} \leavevmode
\label{c-otherstats} 
\begin{enumerate}
\item [\normalfont{(a)}] ${\displaystyle A(t;x)=\frac{\frac{\partial}{\partial x}V(t,(1+t)/2;x)}{1-tV(t,(1+t)/2;x)}}$
\item [\normalfont{(b)}] ${\displaystyle P^{\pk}(t;x)=\frac{\frac{\partial}{\partial x}V(t,1;x)}{1-tV(t,1;x)}}$
\item [\normalfont{(c)}] ${\displaystyle P^{\ddes}(t;x)=\frac{\frac{\partial}{\partial x}V(1,(1+t)/2;x)}{1-V(1,(1+t)/2;x)}}$
\end{enumerate}
\end{cor}
Further specializing Corollary \ref{c-otherstats} (a)
at $t=-1$ implies Theorem \ref{t-main} (a).

Next we show that the $(-1)$-evaluation of the double descent distribution
over $\mathfrak{S}_{2n+1}$ gives the tangent number $E_{2n+1}$.
\begin{thm}
\label{t-ddesformula} For all $n\geq1$, we have
\[
P_{n}^{\ddes}(-1)=\begin{cases}
E_{n}, & \mbox{if }n\text{ is odd,}\\
0, & \mbox{if }n\text{ is even.}
\end{cases}
\]
Thus $P^{\ddes}(-1;x)=\tan(x)$.
\end{thm}
\begin{proof}
Comparing Theorem \ref{t-pkdblV} with Corollary \ref{c-otherstats} (c), we have $P^{\ddes}(-1;x)=P^{(\pk,\dbl)}(1,0;x)$,
which implies $P_{n}^{\ddes}(-1)=P_{n}^{(\pk,\dbl)}(1,0)$ for all
$n\geq1$. Observe that $P_{n}^{(\pk,\dbl)}(1,0)$ is the number of
permutations in $\mathfrak{S}_{n}$ with no double ascents or double
descents. It is easy to see that there are no such permutations for
even $n$, and the only such permutations for odd $n$ are \textit{alternating
permutations}: permutations $\pi=\pi_{1}\pi_{2}\cdots\pi_{n}$ satisfying
$\pi_{1}<\pi_{2}>\pi_{3}<\pi_{4}>\cdots<\pi_{n}$. It is well known
that there are $E_{n}$ alternating permutations in $\mathfrak{S}_{n}$,
and the proof follows.
\end{proof}
Similar reasoning can be used to prove the formula (\ref{e-A-1})
for Eulerian polynomials evaluated at $t=-1$.

\section{\label{s-cyc}Counting derangements by cyclic statistics}

Recall that a \textit{derangement} is a permutation with no fixed
points, i.e., a permutation for which $\pi_{i}\neq i$ for all $i$.
Let $\mathfrak{D}_{n}$ be the set of derangements in $\mathfrak{S}_{n}$.
Our goal in this section is to provide an analogous treatment of the
material from the previous section but for counting derangements with
respect to several ``cyclic statistics'' that we will define shortly.

When writing permutations in cycle notation, we adopt the convention
of writing each cycle with its largest letter in the first position,
and writing the cycles from left-to-right in increasing order of their
largest letters. (This convention is sometimes called \textit{canonical
cycle representation}.) For example, the permutation $\pi=649237185$
in one-line notation is written as $\pi=(42)(716)(8)(953)$ in cycle
notation.

Given $\pi=\pi_{1}\pi_{2}\cdots\pi_{n}$, we say that $\pi_{i}$ is:
\begin{itemize}
\item a \textit{cyclic peak} if $i<\pi_{i}>\pi_{\pi_{i}}$;
\item a \textit{cyclic valley} if $i>\pi_{i}<\pi_{\pi_{i}}$;
\item a \textit{cyclic double ascent} if $i<\pi_{i}<\pi_{\pi_{i}}$;
\item a \textit{cyclic double descent} if $i>\pi_{i}>\pi_{\pi_{i}}$.
\end{itemize}
Every letter of a derangement is either a cyclic peak, cyclic valley,
cyclic double ascent, or cyclic double descent. Define $\cpk(\pi)$,
$\cval(\pi)$, $\cdasc(\pi)$, and $\cddes(\pi)$ to be the number
of cyclic peaks, cyclic valleys, cyclic double ascents, and cyclic
double descents of $\pi$, respectively. 

These ``cyclic statistics'' were studied earlier by, e.g., Zeng \cite{Zeng1993}, Shin and
Zeng \cite{Shin2012}, and Sun and Wang \cite{Sun2014}.\footnote{Chow et al.\ \cite{Chow2014} also derived various formulas for counting
permutations by cyclic peaks and cyclic valleys, but their definitions
for these statistics differ from ours in that they do not allow the
first or last letter of a cycle to be a cyclic peak or cyclic valley.} These statistics are also closely related to a classical permutation
statistic, the excedance number. We say that $i\in[n]$ is an \textit{excedance}
of $\pi$ if $i<\pi_{i}$ and let $\exc(\pi)$ denote the number of
excedances of $\pi$. Then $i$ is an excedance of $\pi$ if and only
if $i$ is a cyclic valley or cyclic double ascent of $\pi$, and
it is well-known that the excedance number $\exc$ and the descent
number $\des$ are equidistributed over $\mathfrak{S}_{n}$.

Define the map $o\colon\mathfrak{S}_{n}\rightarrow\mathfrak{S}_{n}$,
where the input is a permutation in canonical cycle representation
and the output is a permutation in one-line notation, by erasing the
parentheses. Continuing the example with $\pi=(42)(716)(8)(953)$,
we have $o(\pi)=427168953$. The map $o$ is often called Foata's \textit{transformation fondamentale} and first appeared in \cite{Foata1965} (see also \cite{Foata1970}). It is easy to see that the transformation fondamentale is a
bijection; we can recover the cycles of $\pi$ by noting the left-to-right
maxima of $o(\pi)$: given a permutation $\sigma=\sigma_1 \sigma_2 \cdots \sigma_n$, we say that $\sigma_{i}$ is a \textit{left-to-right maximum} of $\sigma$ if $\sigma_{j}<\sigma_{i}$
for all $1\leq j<i$.  

Our work in this section will rely on a cyclic variant of
valley-hopping introduced in \cite{Sun2014}. Define $\theta_{k}\colon\mathfrak{D}_{n}\rightarrow\mathfrak{D}_{n}$
by $\theta_{k}(\pi)=o^{-1}(\varphi_{k}(o(\pi)))$, where the $0$th
letter of $o(\pi)$ is treated as 0 rather than $\infty$. Similarly,
for a subset $S\subseteq[n]$, define $\theta_{S}\colon\mathfrak{D}_{n}\rightarrow\mathfrak{D}_{n}$
by $\theta_{S}=\prod_{k\in S}\theta_{k}$. Then the \textit{cyclic
modified Foata\textendash Strehl action} (or \textit{cyclic valley-hopping})
is the $\mathbb{Z}_{2}^{n}$-action defined by the involutions $\theta_{S}$.
It is easy to see that cyclic valley-hopping toggles between cyclic
double ascents and cyclic double descents, but does not change cyclic
peaks or cyclic valleys.

Let $\widetilde{\mathfrak{D}}_{n}$ denote the set of derangements of
$[n]$ with no cyclic double ascents and where each cyclic double
descent is assigned one of two colors: red or blue. Then cyclic valley-hopping
induces a map $\mathring{\Phi}$ from $\widetilde{\mathfrak{D}}_{n}$
to $\mathfrak{D}_{n}$ defined in the analogous way as the map $\Phi$
from Section 2, but with $R(\pi)$ being the set of red cyclic double
descents. It then follows from the same reasoning as in the proof
of Lemma \ref{l-pkdblbij} that $\mathring{\Phi}$ is a $(\cpk,\cdbl)$-preserving
bijection, where $\cdbl(\pi)\coloneqq\cdasc(\pi)+\cddes(\pi)$ is
the total number of cyclic double ascents and cyclic double descents
of $\pi$.

\subsection{A cyclic analogue of Theorem 3 for derangements}

For permutation statistics $\st_{1},\st_{2},\dots,\st_{m}$ and variables
$t_{1},t_{2},\dots,t_{m}$, we define the polynomials $\{D_{n}^{(\st_{1},\st_{2},\dots\st_{m})}(t_{1},t_{2},\dots,t_{m})\}_{n\geq0}$
by 
\[
D_{n}^{(\st_{1},\st_{2},\dots\st_{m})}(t_{1},t_{2},\dots,t_{m})\coloneqq\sum_{\pi\in\mathfrak{D}_{n}}t_{1}^{\st_{1}(\pi)}t_{2}^{\st_{2}(\pi)}\cdots t_{m}^{\st_{m}(\pi)}.
\]
and we let
\[
D^{(\st_{1},\st_{2},\dots\st_{m})}(t_{1},t_{2},\dots,t_{m};x)\coloneqq1+\sum_{n=1}^{\infty}D_{n}^{(\st_{1},\st_{2},\dots\st_{m})}(t_{1},t_{2},\dots,t_{m})\frac{x^{n}}{n!}
\]
be their exponential generating function.\footnote{As before, if we have a single statistic $\st$, we write these simply
as $D_{n}^{\st}(t)$ and $D^{\st}(t;x)$.} 
These encode the distributions of permutation statistics over derangements.

We now present a cyclic analogue of Theorem~\ref{t-pkdblV} for derangements.
\begin{thm}
\label{t-cpkcdblV} ${\displaystyle D^{(\cpk,\cdbl)}(s,t;x)=\frac{1}{1-sV(s,t;x)}}$
\end{thm}

\begin{proof}
It suffices to show that 
\begin{align}
D_{n}^{(\cpk,\cdbl)}(s,t) & =\sum_{k=0}^{\left\lfloor n/2\right\rfloor }s^{k}\sum_{B}U_{\left|B_{1}\right|-2}(s,t)\cdots U_{\left|B_{k}\right|-2}(s,t)\label{e-bigsumD}
\end{align}
where the second sum is over all ordered set partitions $B$ of $[n]$
into parts $B_{1},\dots,B_{k}$. We interpret the right-hand side
of (\ref{e-bigsumD}) as in the proof of Theorem \ref{t-pkdblV} (without the initial block $B_0$ with an $\infty$) and
apply the same sign-reversing involution; the objects that remain
after cancellation are of the form 
\[
c_{1}>c_{2}>\cdots>c_{\left|B_{0}\right|}\mid c_{\left|B_{0}\right|+1}>c_{\left|B_{0}\right|+2}>\cdots>c_{\left|B_{0}\right|+\left|B_{1}\right|}\mid\cdots\mid c_{n-\left|B_{k}\right|+1}>\cdots>c_{n}
\]
with no dominoes and such that $c_{i}<c_{i+1}$ whenever $c_{i}$
and $c_{i+1}$ are in separate blocks.

Now, rather than treating these remaining objects as permutations
in one-line notation, we want to treat them as permutations in cycle
notation with blocks corresponding to cycles. In doing so, we merge
two adjacent blocks whenever the first element of the second block
in the pair is not larger than all elements from all preceding blocks,
i.e., whenever the element is not a left-to-right maximum of the underlying
permutation written in one-line notation; this guarantees that the
resulting permutations are correctly written in canonical cycle representation and is clearly reversible. 

Moreover, these permutations are derangements because each block has
size at least 2, and they have no cyclic double ascents and have each
cyclic double descent colored either red or blue (depending on the
color of the corresponding monomino). In other words, these permutations
$\pi$ are precisely the elements of $\widetilde{\mathfrak{D}}_{n}$ and
each contributes a weight of $s^{\cpk(\pi)}t^{\cdbl(\pi)}$ to the
right-hand side of Equation (\ref{e-bigsumD}). The result then follows
from applying the $(\cpk,\cdbl)$-preserving bijection $\mathring{\Phi}$.
\end{proof}

We extend Theorem \ref{t-cpkcdblV} to an analogous result
for the joint distribution of the statistics $\cpk$, $\cval$, $\cdasc$, and $\cddes$ over $\mathfrak{D}_n$.

\begin{thm}
\label{t-cyc4stats} ${\displaystyle {\displaystyle D^{(\cpk,\cval,\cdasc,\cddes)}(s,t,u,v;x)}=\frac{1}{1-stV(st,(u+v)/2;x)}}$
\end{thm}
\begin{proof}
First, observe that $\cpk(\pi)=\cval(\pi)$ for all derangements $\pi$.
The cyclic valley-hopping map $\theta_{S}$, where $S$ is the set
containing all cyclic double ascents and cyclic double descents of
$\pi$, is a $(\cpk,\cval)$-preserving involution on $\mathfrak{D}_{n}$
that switches cyclic double ascents with cyclic double descents. Thus,
we have
\[
D_{n}^{(\cpk,\cval,\cdasc,\cddes)}(s,t,u,v)=D_{n}^{(\pk,\dbl)}\left(st,\frac{1}{2}(u+v)\right)
\]
for all $n\geq1$, which along with Theorem \ref{t-cpkcdblV} proves
the result.
\end{proof}

We can use Theorem \ref{t-cyc4stats} and Equation (\ref{e-Vcf}) to derive the formula 
\begin{equation}
{D^{(\cpk,\cval,\cdasc,\cddes)}(s,t,u,v;x)}=\frac{\alpha e^{-\frac{1}{2}(u+v)x}}{\alpha\cosh(\frac{1}{2}\alpha x)-(u+v)\sinh(\frac{1}{2}\alpha x)} \label{e-4stats}
\end{equation}
where $\alpha=\sqrt{(u+v)^{2}-4st}$.
Furthermore, given any permutation $\pi$, let $\cyc(\pi)$ denote the number of cycles of $\pi$ and let $\fix(\pi)$ be the number of fixed points of $\pi$. Standard applications of the exponential formula (see \cite[Section 5.1]{Stanley2001}) yield the identities
\begin{equation}
D^{(\cpk,\cval,\cdasc,\cddes,\cyc)}(s,t,u,v,w;x)=
D^{(\cpk,\cval,\cdasc,\cddes)}(s,t,u,v;x)^w \label{e-5stats}
\end{equation}
and
\begin{equation}
1+P^{(\cpk,\cval,\cdasc,\cddes,\cyc,\fix)}(s,t,u,v,w,y;x)=
e^{wyx}D^{(\cpk,\cval,\cdasc,\cddes,\cyc)}(s,t,u,v,w;x). \label{e-6stats}
\end{equation}
Then, combining Equations (\ref{e-4stats}), (\ref{e-5stats}), and (\ref{e-6stats}) yields the following exponential generating function formula for the sextuple distribution $(\cpk,\cval,\cdasc,\cddes,\cyc,\fix)$ over all permutations:
\[
1+P^{(\cpk,\cval,\cdasc,\cddes,\cyc,\fix)}(s,t,u,v,w,y;x) =\left(\frac{\alpha e^{(y-\frac{1}{2}(u+v))x}}{\alpha\cosh(\frac{1}{2}\alpha x)-(u+v)\sinh(\frac{1}{2}\alpha x)}\right)^{w}.
\]
An equivalent form of this formula was proven earlier by Zeng \cite[Th\'{e}or\`{e}me 1]{Zeng1993}.

\subsection{Counting derangements by excedances}

In the remainder of this section, we examine specializations of Theorem \ref{t-cyc4stats} that give rise to formulas
for individual cyclic statistics, beginning with the excedance number.

The excedance polynomials $D_{n}(t)\coloneqq D_{n}^{\exc}(t)$ have been well-studied; for example, it is known that they have exponential generating function 
\begin{align*}
D(t;x) & \coloneqq D^{\exc}(t;x)=\frac{(1-t)e^{-x}}{e^{-(1-t)x}-t}
\end{align*}
\cite{Roselle1968} and are $\gamma$-positive \cite{Athanasiadis2011/12,Shin2012,Sun2014}. From Theorem~\ref{t-cyc4stats} we obtain the following.
\begin{cor}
\label{c-exc}${\displaystyle {\displaystyle D(t;x)}=\frac{1}{1-tV(t,(1+t)/2;x)}}$
\end{cor}

It follows from Corollary \ref{c-exc} that the exponential generating
function for the excedance polynomials $D(t;x)$ evaluated at $t=-1$
is the reciprocal of the exponential generating function $F(x)=\cosh(x)$
for the sequence $1,0,1,0,\dots$.
\begin{cor}
$D(-1;x)=F(x)^{-1}$
\end{cor}

This identity can be used to rederive the classical result due to
Roselle \cite{Roselle1968} that 
\[
D_{n}(-1)=\begin{cases}
(-1)^{n/2}E_{n}, & \mbox{if }n\text{ is even,}\\
0, & \mbox{if }n\text{ is odd.}
\end{cases}
\]

We note that, from Equation~\eqref{e-Ust_rec}, one can show that
\begin{align}
	U_{n}(t,(1+t)/2)=1+t+\cdots+t^{n}\label{e-Wsimple}
\end{align}
for $n\geq1$. Then, using Equation~\eqref{e-Wsimple}, Corollary \ref{c-exc} can be obtained as a specialization of a formula of Shareshian and Wachs \cite[Equation (1.4)]{Shareshian2010} involving Eulerian quasisymmetric functions.

\subsection{Counting derangements by cyclic peaks}

Next, we examine the distribution of the cyclic peak number over derangements.

\begin{cor}
\label{c-cpk} ${\displaystyle D^{\cpk}(t;x)=\frac{1}{1-tV(t,1;x)}}$
\end{cor}

By specializing (\ref{e-4stats}) appropriately, we derive the formula
\[
D^{\cpk}(t;x)=\frac{\sqrt{1-t}e^{-x}}{\sqrt{1-t}\cosh(x\sqrt{1-t})-\sinh(x\sqrt{1-t})}.
\]
The first several polynomials $D_{n}^{\cpk}(t)$ are given in the following table.
\begin{center}
\begin{tabular}{ccc}
	\begin{tabular}{c|c}
		$n$ & $D_{n}^{\cpk}(t)$\tabularnewline
		\hline 
		1 & 0\tabularnewline
		2 & $t$\tabularnewline
		3 & $2t$\tabularnewline
		4 & $4t+5t^{2}$\tabularnewline
	\end{tabular}
	& &
	\begin{tabular}{c|c}
		$n$ & $D_{n}^{\cpk}(t)$\tabularnewline
		\hline 
		5 & $8t+36t^{2}$\tabularnewline
		6 & $16t+188t^{2}+61t^{3}$\tabularnewline
		7 & $32t+864t^{2}+958t^{3}$\tabularnewline
		8 & $64t+3728t^{2}+9656t^{3}+1385t^{4}$\tabularnewline
	\end{tabular}
\end{tabular}
\end{center}

Note that the coefficient of $t$ in $D_{n}^{\cpk}(t)$ seems to be
$2^{n-2}$ for all $n\geq2$; this is easy to explain combinatorially.
\begin{prop}
For all $n\geq2$, the number of derangements in $\mathfrak{D}_{n}$
with exactly one cyclic peak is $2^{n-2}$.
\end{prop}

\begin{proof}
It is easy to see that every derangement $\pi$ of $[n]$ with exactly
one cyclic peak has exactly one cycle and can be written in the form
\[
(c_{1}c_{2}\cdots c_{k}c_{k+1}\cdots c_{n})
\]
where $c_{1}=n$ is the only cyclic peak of $\pi$, the sequence $c_{2}\cdots c_{k}$
is decreasing (with $c_{k}=1)$, and the sequence $c_{k+1}\cdots c_{n}$
is increasing. Thus, for every letter $i$ between $2$ and $n-1$,
either $i$ belongs to the decreasing sequence or the increasing sequence,
and these $n-2$ choices completely determine the derangement $\pi$.
\end{proof}
It follows from Corollary \ref{c-cpk} that the exponential generating
function for the numbers $D_{n}^{\cpk}(-1)$ is the reciprocal of the exponential generating function $G(x)$ for
the Pell numbers.
\begin{cor}
$D^{\cpk}(-1;x)=G(x)^{-1}$
\end{cor}

We do not have a combinatorial interpretation for
the numbers $D_{n}^{\cpk}(-1)$ themselves. The first several of these
numbers appear in the following table.
\noindent \begin{center}
\begin{tabular}{c|cccccccccc}
$n$ & 1 & 2 & 3 & 4 & 5 & 6 & 7 & 8 & 9 & 10\tabularnewline
\hline 
$D_{n}^{\cpk}(-1)$ & $0$ & $-1$ & $-2$ & $1$ & $28$ & $111$ & $-126$ & $-4067$ & $-26280$ & $53663$\tabularnewline
\end{tabular}
\end{center}

\subsection{Counting derangements by cyclic double descents}

Lastly, we study the distribution of the cyclic double descent number
over derangements.

\begin{cor}
\label{c-cddes} ${\displaystyle D^{\cddes}(t;x)=\frac{1}{1-V(1,(1+t)/2;x)}}$
\end{cor}

The exponential generating function formula
\[
D^{\cddes}(t;x)=\frac{\beta e^{-\frac{1}{2}(1+t)x}}{\beta\cosh(\frac{1}{2}\beta x)-(1+t)\sinh(\frac{1}{2}\beta x)},
\]
where $\beta=\sqrt{(t+3)(t-1)}$, can be obtained by specializing
(\ref{e-4stats}). The first several of the polynomials $D_{n}^{\cddes}(t)$ appear in the following table.
\noindent \begin{center}
\begin{tabular}{ccc}
	\begin{tabular}{c|c}
		$n$ & $D_{n}^{\cddes}(t)$\tabularnewline
		\hline 
		1 & 0\tabularnewline
		2 & 1\tabularnewline
		3 & $1+t$\tabularnewline
		4 & $6+2t+t^{2}$\tabularnewline
	\end{tabular}
	&&
	\begin{tabular}{c|c}
		$n$ & $D_{n}^{\cddes}(t)$\tabularnewline
		\hline 
		5 & $19+21t+3t^{2}+t^{3}$\tabularnewline
		6 & $109+98t+53t^{2}+4t^{3}+t^{4}$\tabularnewline
		7 & $588+808t+334t^{2}+118t^{3}+5t^{4}+t^{5}$\tabularnewline
		8 & $4033+5766t+3827t^{2}+952t^{3}+248t^{4}+6t^{5}+t^{6}$\tabularnewline
	\end{tabular}
\end{tabular}
\end{center}

An \textit{increasing run} of a permutation $\pi$ is a maximal increasing
consecutive subsequence of $\pi$ (in one-line notation). For example,
the increasing runs of $\pi=467192685$ are $467$, $19$, $268$,
and $5$. We call increasing runs of length 1 \textit{short runs}. The
sequence of constant coefficients of $D_{n}^{\cddes}(t)$ matches
the OEIS sequence \cite[A097899]{oeis} for the number of permutations
of $[n]$ with no short runs; this can be verified by comparing the
evaluation $D^{\cddes}(0;x)$ with the exponential generating function
of this OEIS entry, but we give a bijective proof below. 

For the purpose of this proof, let us temporarily modify our convention for cycle notation so that we write each cycle with its smallest letter in the first position, and write the cycles from left-to-right in decreasing order of their smallest letters. For example, whereas we previously wrote $\pi=649237185$ as $\pi=(42)(716)(8)(953)$ in canonical cycle representation, now we write $\pi$ as $\pi=(8)(395)(24)(167)$. Let $o'\colon\mathfrak{S}_{n}\rightarrow\mathfrak{S}_{n}$ be the map defined by taking a permutation in cycle notation under this new convention and erasing the parentheses, yielding a permutation in one-line notation. This is a bijection; we can recover the cycles of $\pi$ by noting the left-to-right minima of $o(\pi)$: given a permutation $\sigma=\sigma_1 \sigma_2 \cdots \sigma_n$, we say that $\sigma_{i}$ is a \textit{left-to-right minimum} of $\sigma$ if $\sigma_{j}>\sigma_{i}$
for all $1\leq j<i$.

\begin{prop}
An letter $i\in[n]$ is a fixed point or cyclic double descent of $\pi\in\mathfrak{S}_{n}$
if and only if $i$ is a short run of $o^{\prime}(\pi)$. 
\end{prop}
In particular, this proposition implies that the number of derangements
of $[n]$ with no cyclic double descents is equal to the number of
permutations of $[n]$ with no short runs.
\begin{proof}
Fix a permutation $\pi\in\mathfrak{S}_{n}$ and a letter $i\in [n]$. Let us write $o^{\prime}(\pi)=\sigma_{1}\sigma_{2}\cdots\sigma_{n}$
in one-line notation, and take $\sigma_{j}=i$. By taking as convention
$\sigma_{0}=\infty$ and $\sigma_{n+1}=0$, it is easy to see that
$\sigma_{j}=i$ is a short run of $o^{\prime}(\pi)$ if and only if
$\sigma_{j-1}>\sigma_{j}>\sigma_{j+1}$.

We divide into cases. First, suppose that $i\in[n]$ is a fixed point
of $\pi$. Then $\sigma_{j}$ and $\sigma_{j+1}$
are both left-to-right minima, so $\sigma_{j-1}>\sigma_{j}>\sigma_{j+1}$.
Now suppose that $i\in[n]$ is a cyclic double descent of $\pi$.
Note that the first letter of a cycle (under our current convention)
cannot be a cyclic double descent. If $i$ is neither the first nor
last letter of its cycle in $\pi$, then $\sigma_{j-1}>\sigma_{j}>\sigma_{j+1}$.
Otherwise, if $i$ is the last letter of its cycle in $\pi$, then
$\sigma_{j-1}>\sigma_{j}$ and $\sigma_{j+1}$ is a left-to-right
minimum; thus $\sigma_{j-1}>\sigma_{j}>\sigma_{j+1}$.
In each case, it follows that $\sigma_{j}=i$ is a short run of $o^{\prime}(\pi)$.
Hence, every fixed point and cyclic double descent of $\pi$ is
a short run of $o^{\prime}(\pi)$; the reverse direction is similar.
\end{proof}

Finally, we give a cyclic analogue of Theorem \ref{t-ddesformula}
for derangements.
\begin{thm}
	\label{t-cddesformula} For all $n\geq1$, we have
	\[
	D_{n}^{\cddes}(-1)=\begin{cases}
	E_{n}, & \mbox{if }n\text{ is even,}\\
	0, & \mbox{if }n\text{ is odd.}
	\end{cases}
	\]
	Thus $D^{\cddes}(-1;x)=\sec(x)$.
\end{thm}

In other words, the $(-1)$-evaluation of the cyclic double descent
distribution over $\mathfrak{D}_{2n}$ gives the secant number $E_{2n}$.
\begin{proof}
By comparing Theorem \ref{t-cpkcdblV} with Corollary \ref{c-cddes}, 
we have $D_{n}^{\cddes}(-1)=D_{n}^{(\cpk,\cdbl)}(1,0)$ for
	all $n\geq1$. Observe that $D_{n}^{(\cpk,\cdbl)}(1,0)$ is the number
	of derangements of $[n]$ with no cyclic double ascents or cyclic
	double descents. Because the number of cyclic peaks of any permutation
	is equal to its number of cyclic valleys, it is evident that there
	are no such permutations for odd $n$, and it is easy to see that
	the map $o$ defined earlier is a bijection between such permutations
	for an even $n$ and \textit{reverse-alternating permutations} of
	$[n]$: permutations $\pi=\pi_{1}\pi_{2}\cdots\pi_{n}$ satisfying
	$\pi_{1}>\pi_{2}<\pi_{3}>\pi_{4}<\cdots>\pi_{n}$. Since there are
	$E_{n}$ reverse-alternating permutations in $\mathfrak{S}_{n}$,
	the proof follows.
\end{proof}

%%%%%%%%%%%%%%%%%%%%%%%%%%%%%%%%%%%%%%%%%%%%%%%%%%%%%%%
\subsection*{Acknowledgements}

We thank an anonymous referee for providing thoughtful comments which improved the presentation of this paper and for pointing us to several key references from the literature on Chebyshev polynomials. We also thank Mohamed Omar and Justin Troyka for helpful discussions relating to this project.

%%%%%%%%%%%%%%%%%%%%%%%%%%%%%%%%%%%%%%%%%%%%%%%%%%%%%%%
% You do not have to use the same format for your references, but 
%    include everything in this file.  Don't use natbib please.
% If you use BibTeX to create a bibliography, copy the .bbl file into here.

\bibliographystyle{plain}
\bibliography{bibliography}

\end{document}